\documentclass{amsart}
\usepackage{amsmath}
\usepackage{amssymb}
\usepackage{amsfonts}

\newtheorem{theorem}{Theorem}
\theoremstyle{plain}

\newtheorem{corollary}{Corollary}

\newtheorem{definition}{Definition}

\newtheorem{notation}{Notation}

\newtheorem{remark}{Remark}

\numberwithin{equation}{section}
\input{tcilatex}

\begin{document}
\title[Nonlinear Operators, Fixed-Point Theorems, Nonlinear Equations]{On
some nonlinear operators, fixed-point theorems and nonlinear equations}
\author{Kamal N. Soltanov}
\address{{\small Dep. Math. Fac. Sci. Hacettepe University, Beytepe, Ankara,
TR-06800, TURKEY ; }\\
{\small Tel. +90 312 2977860; Fax: +90 312 2992017\ }}
\email{soltanov@hacettepe.edu.tr ; sultan\_kamal@hotmail.com\ }
\urladdr{http://www.mat.hacettepe.edu.tr/personel/akademik/ksoltanov/index.html }

\begin{abstract}
In this article we discuss the solvability of some class of fully nonlinear
equations, and equations with p-Laplacian in more general conditions by
using a new approach given in [1] for studying the nonlinear continuous
operator. Moreover we reduce certain general results for the continuous
operators acting on Banach spaces, and investigate their image. Here we also
consider the existence of a fixed-point of the continuous operators under
various conditions.
\end{abstract}

\thanks{Supported by 110T558-projekt of TUBITAK}
\maketitle

\section{Introduction}

In the present paper we consider the boundary-value problem for the fully
nonlinear equation of the second order 
\begin{equation}
F\left( x,u,Du,\Delta u\right) =h\left( x\right) ,\quad x\in \Omega , 
\tag{1.1}
\end{equation}%
and also for the nonlinear equations with $p$-Laplacian that depend upon the
parameters $\lambda $ and $\mu $ 
\begin{equation}
-\nabla \left( \left\vert \nabla u\right\vert ^{p-2}\nabla u\right) +G\left(
x,u,Du,\lambda ,\mu \right) =h\left( x\right) ,\quad x\in \Omega ,  \tag{1.2}
\end{equation}%
on the smooth bounded domain $\Omega \subset 
\mathbb{R}
^{n}$ ($n\geq 1$), where $F\left( x,\xi ,\eta ,\zeta \right) $ and $G\left(
x,\xi ,\eta ,\lambda ,\mu \right) $ are Caratheodory functions. We deal with
the properties of the nonlinear operators generated by the posed problems
and study the solvability of these problems by using the general results of
such type as in [1]. It should be noted that equations of such type arise in
the diffusion processes, reaction-diffusion processes etc., in the
steady-state case (see, for example [2 - 13] and their references).
Furthermore we discuss some of the nonlinear continuous mappings acting on
Banach spaces and an equation (inclusion) with mappings of such type.

The problems of such type were studied earlier under various conditions in
the semilinear ([2, 3, 7 - 11, 13] etc.) and in the fully nonlinear cases
([4, 12, 14] etc.). In the mentioned articles, the known general results
having such conditions that cannot be applicable to the problems considered
here, were used, whereas in this article we want to investigate these
problems under more general conditions. Therefore we need to use a general
result that will be applicable to the considered problems here and
consequently, the conditions of the general result differ from the
conditions of the known results. The results of [1] and the general results
adduced here allow us to study the imposed problem in more general
conditions.

So for our goal we lead to a fixed-point theorem for nonlinear continuous
mappings acting on Banach spaces and also a solvability theorem for
nonlinear equations involving continuous operators. These general results
allow us to study various nonlinear mappings and also nonlinear problems
under more general conditions. Therefore we investigate boundary value
problems (BVP) for the nonlinear differential equations by using the
mentioned general results. Moreover we obtain the existence of the fixed
point for the operator associated with imposed problem. We also note that
Theorem 1 of present article is a result of the type of Lax-Milgram theorem.

This article is constructed in the following way. In section 2, we lead a
general theorem that shows how one can locally determine the image of a
subset of the domain of definition under the continuous mapping acting in
Banach space (cf. [1, 8, 11 - 14] etc.). From this theorem we deduce the
solvability theorem and an existence of the fixed point of the continuous
mapping under the corresponding conditions, and consequently, an existence
of the fixed point of the mapping associated with the problem, considered in
the next section. In section 3, we prove a solvability theorem for an
equation with the perturbed operator on Banach spaces, and then in sections
4 and 5 by using these results we study the solvability of the boundary
value problem for various classes of nonlinear differential equations.

\section{\protect\bigskip Some General Results on Solvability}

Let $X,Y$ \ be reflexive Banach spaces and $X^{\ast },Y^{\ast }$ be their
dual spaces, moreover let $Y$ be reflexive with strictly convex norm
together with $Y^{\ast }$ (this condition is not complementary condition;
see, for example, [15]) and $f:D\left( f\right) \subseteq X\longrightarrow Y$
be an operator.

So, we will conduct here the special case of the main result of [1].
Consider the following conditions. Let the closed ball $B_{r_{0}}^{X}\left(
0\right) $ of $X$ be contained in $D\left( f\right) $, i.e. $%
B_{r_{0}}^{X}\left( 0\right) \subseteq D\left( f\right) $ $\subset X$ and on 
$B_{r_{0}}^{X}\left( 0\right) $ are fulfilled the conditions:

(i) $f:B_{r_{0}}^{X}\left( 0\right) \subseteq D\left( f\right) \subseteq
X\longrightarrow Y$ be a continuous operator that bounded on $%
B_{r_{0}}^{X}\left( 0\right) $, i.e.%
\begin{equation*}
\left\Vert f\left( x\right) \right\Vert _{Y}\leq \mu \left( \left\Vert
x\right\Vert _{X}\right) ,\quad \forall x\in B_{r_{0}}^{X}\left( 0\right) ;
\end{equation*}

(ii) there is a mapping $g:D\left( g\right) \subseteq X\longrightarrow
Y^{\ast }$ such that $D\left( f\right) \subseteq D\left( g\right) $, and for
any $S_{r}^{X}\left( 0\right) \subset B_{r_{0}}^{X}\left( 0\right) $, $%
0<r\leq r_{0}$, cl\ $g\left( S_{r}^{X}\left( 0\right) \right) =\overline{%
g\left( S_{r}^{X}\left( 0\right) \right) }\equiv S_{r}^{Y^{\ast }}\left(
0\right) $, $S_{r}^{X}\left( 0\right) \subseteq g^{-1}\left( S_{r}^{Y^{\ast
}}\left( 0\right) \right) $ 
\begin{equation}
\left\langle f\left( x\right) ,g\left( x\right) \right\rangle \geq \nu
\left( \left\Vert x\right\Vert _{X}\right) \left\Vert x\right\Vert _{X},%
\text{ a.e. }x\in B_{r_{0}}^{X}\left( 0\right) \quad \&\ \nu \left(
r_{0}\right) \geq \delta _{0}>0  \tag{2.1}
\end{equation}%
holds\footnote{%
In particular, the mapping $g$ can be a linear bounded operator as $g\equiv
L:X\longrightarrow Y^{\ast }$ that satisfies the conditions of \textit{(ii). 
}}, where $\mu :R_{+}^{1}\longrightarrow R_{+}^{1}$ and $\nu
:R_{+}^{1}\longrightarrow R^{1}$ are continuous functions ( $\mu ,\nu \in
C^{0}$), moreover $\nu $ is the nondecreasing function for $\tau \geq \tau
_{0}$, $r_{0}\geq \tau _{0}\geq 0$; $\tau _{0},\delta _{0}>0$ are constants;

(iii) almost each $\widetilde{x}\in intB_{r_{0}}^{X}\left( 0\right) $
possesses a neighborhood $V_{\varepsilon }\left( \widetilde{x}\right) $, $%
\varepsilon \geq \varepsilon _{0}>0$ such that the following inequality 
\begin{equation}
\left\Vert f\left( x_{2}\right) -f\left( x_{1}\right) \right\Vert _{Y}\geq
\Phi \left( \left\Vert x_{2}-x_{1}\right\Vert _{X},\widetilde{x},\varepsilon
\right) ,  \tag{2.2}
\end{equation}%
holds for any $\forall x_{1},x_{2}\in V_{\varepsilon }\left( \widetilde{x}%
\right) \cap B_{r_{0}}^{X}\left( 0\right) $, where $\Phi \left( \tau ,%
\widetilde{x},\varepsilon \right) \geq 0$ is a continuous function at $\tau $
and $\Phi \left( \tau ,\widetilde{x},\varepsilon \right) =0\Leftrightarrow
\tau =0$ (in particular, $\widetilde{x}=0$, $\varepsilon =\varepsilon
_{0}=r_{0}$ and $V_{\varepsilon }\left( \widetilde{x}\right)
=V_{r_{0}}\left( 0\right) \equiv B_{r_{0}}^{X}\left( 0\right) $, consequently%
$\ \Phi \left( \tau ,\widetilde{x},\varepsilon \right) \equiv \Phi \left(
\tau ,0,r_{0}\right) $ on $B_{r_{0}}^{X}\left( 0\right) $),

\begin{theorem}
Let $X,Y$ be Banach spaces such as above and $f:D\left( f\right) \subseteq
X\longrightarrow Y$ be an operator. Assume that on the closed ball $%
B_{r_{0}}^{X}\left( 0\right) $ $\subseteq D\left( f\right) $ $\subset X$ the
conditions (i) and (ii) are fulfilled then the image $f\left(
B_{r_{0}}^{X}\left( 0\right) \right) $ of the ball $B_{r_{0}}^{X}\left(
0\right) $ contains an everywhere dense subset of $M$ that has the form 
\begin{equation*}
M\equiv \left\{ y\in Y\left\vert \ \left\langle y,g\left( x\right)
\right\rangle \leq \left\langle f\left( x\right) ,g\left( x\right)
\right\rangle ,\right. \forall x\in S_{r_{0}}^{X}\left( 0\right) \right\} .
\end{equation*}

Furthermore if in addition the image $f\left( B_{r_{0}}^{X}\left( 0\right)
\right) $ of the ball $B_{r_{0}}^{X}\left( 0\right) $ is closed or the
condition (iii) is fulfilled then the image $f\left( B_{r_{0}}^{X}\left(
0\right) \right) $ is a bodily subset of $Y$, moreover $f\left(
B_{r_{0}}^{X}\left( 0\right) \right) $ contains the above bodily subset $M$.
\end{theorem}

The proof of this theorem is obtained from general result that was proven in
[14] (see also, [1]).\footnote{%
We note that Theorem 1 is a generalization of Theorems of such type from
\par
Soltanov, K.N.: On equations with continuous mappings in Banach spaces.
Funct. Anal. Appl. 33, (1999) 1, 76-81.}

\begin{remark}
\textbf{1. }It is easy to see that the condition $B_{r_{0}}^{X}\left(
0\right) \subseteq D\left( f\right) $ is not essential, because if $D\left(
f\right) $ comprises a bounded closed subset $U\left( x_{0}\right) \subseteq
X$ of some element $x_{0}\in D(f)$ such that $U\left( x_{0}\right) $\ is
topologically equivalent to $B_{1}^{X}\left( 0\right) $ and $U\left(
x_{0}\right) \subseteq D(f)\cap D\left( g\right) $, then we can formulate
the conditions and statement of this theorem analogously, i.e. for this we
determine the operator $\widetilde{f}\left( x\right) =f\left( x\right)
-f\left( x_{0}\right) $ and assume that 
\begin{equation*}
\left\Vert \widetilde{f}\left( x\right) \right\Vert _{X^{\ast }}\leq \mu
\left( \left\Vert x-x_{0}\right\Vert _{X}\right) ,
\end{equation*}%
holds for any $x\in U\left( x_{0}\right) $ and 
\begin{equation}
\left\langle \widetilde{f}\left( x\right) ,g\left( x-x_{0}\right)
\right\rangle \geq \nu \left( \left\Vert x-x_{0}\right\Vert _{X}\right)
\left\Vert x-x_{0}\right\Vert _{X},  \tag{2.1'}
\end{equation}%
holds for almost all $x\in U\left( x_{0}\right) $. Moreover $\nu \left(
\left\Vert x-x_{0}\right\Vert _{X}\right) \geq \delta _{0}>0$ for any $x\in
\partial U\left( x_{0}\right) $, where $g:D\left( g\right) \subseteq
X\longrightarrow Y^{\ast }$ such that $D\left( f\right) \subseteq D\left(
g\right) $ and $g$\ satisfies a claim analogously of the condition (ii)
respect to $U\left( x_{0}\right) $. In this case, we define subset $%
\widetilde{M}_{x_{0}}$ in the form 
\begin{equation}
\widetilde{M}_{x_{0}}\equiv \left\{ y\in Y\left\vert \ 0\leq \left\langle
f\left( x\right) -y,g\left( x-x_{0}\right) \right\rangle ,\right. \forall
x\in \partial U\left( x_{0}\right) \right\} .  \tag{2.2'}
\end{equation}

\textbf{2.} In the formulation of Theorem 1 we use the equality $\left\Vert
g\left( x\right) \right\Vert _{Y^{\ast }}\equiv \left\Vert x\right\Vert _{X}$
that can be determined by the known way, i.e. $g%
{\acute{}}%
\left( x\right) \equiv \frac{\left\Vert x\right\Vert _{X}}{\left\Vert
g\left( x\right) \right\Vert _{Y^{\ast }}}$ $g\left( x\right) $ for any $%
x\in D\left( g\right) \subseteq X$.
\end{remark}

Condition (\textit{iii}) of Theorem 1 can be generalized, for example, as in
the following proposition.

\begin{corollary}
Let all conditions of Theorem 1 be fulfilled except inequality (2.2), and
instead of that, let the following inequality 
\begin{equation}
\left\Vert f\left( x_{2}\right) -f\left( x_{1}\right) \right\Vert _{Y}\geq
\Phi \left( \left\Vert x_{2}-x_{1}\right\Vert _{X},\widetilde{x},\varepsilon
\right) +\psi \left( \left\Vert x_{1}-x_{2}\right\Vert _{Z},\widetilde{x}%
,\varepsilon \right) ,  \tag{2.3}
\end{equation}%
be fulfilled for any $x_{1},x_{2}\in V_{\varepsilon }\left( \widetilde{x}%
\right) \cap B_{r_{0}}^{X}\left( 0\right) $, where $\Phi \left( \tau ,%
\widetilde{x},\varepsilon \right) $ is a function such as in the condition
(iii), $Z$ is a Banach space and the inclusion $X\subset Z$ is compact, and $%
\psi \left( \cdot ,\widetilde{x},\varepsilon \right)
:R_{+}^{1}\longrightarrow R^{1}$ is a continuous function at $\tau $ and $%
\psi \left( 0,\widetilde{x},\varepsilon \right) =0$ . Then the statement of
Theorem 1 is correct.
\end{corollary}

{\large Note}\textbf{.} \textit{It should be noted that this result is a
generalization of the known Lax-Milgram theorem to the nonlinear case in the
class of Banach spaces when all conditions of this theorem are fulfilled on
whole space. Indeed, we can formulate the Lax-Milgram theorem for the linear
operator }$T$\textit{\ acting on the real Hilbert space }$X$\textit{\ in the
form: \ \ \ }

\textit{a) There exists a positive constant }$\gamma $\textit{\ such that }%
\begin{equation*}
\left\vert \left( Tx,y\right) \right\vert \leq \gamma \left\Vert
x\right\Vert _{X}\cdot \left\Vert y\right\Vert _{X},\quad \forall \left(
x,y\right) \in X\times X\ ;
\end{equation*}%
\textit{that is equivalent to boundedness of operator }$T:X\longrightarrow X$%
\textit{.}

\textit{b) there exists a positive constant }$\delta $\textit{\ such that }%
\begin{equation*}
\left( Tx,x\right) \geq \delta \left\Vert x\right\Vert _{X}^{2}\ ;
\end{equation*}%
\textit{that is equivalent to the coerciveness of }$T:X\longrightarrow X$,
i.e. $T$ \textit{satisfies the condition (ii) for the special case (}$%
g\equiv id$\textit{).}

\textit{Then equation }$Tx=y$\textit{\ is solvable for any }$y\in X.$\textit{%
\ \ }

\textit{From the condition b, it follows that }%
\begin{equation*}
\left\Vert T\left( x_{1}-x_{2}\right) \right\Vert _{X}\geq \delta \left\Vert
x_{1}-x_{2}\right\Vert _{X}\ ;
\end{equation*}%
\textit{i.e. inequality (2.2) holds for any }$x_{1},x_{2}\in X$\textit{.}

From Theorem 1 it immediately follows that

\begin{theorem}
(\textbf{Fixed-Point Theorem}). Let $X$ be a reflexive separable Banach
space and $f_{1}:D\left( f_{1}\right) \subseteq X\longrightarrow X$ be a
bounded continuous operator. Moreover, let on a closed ball $%
B_{r_{0}}^{X}\left( x_{0}\right) \subseteq D\left( f_{1}\right) $, where $%
x_{0}\in D\left( f_{1}\right) $, operator $f\equiv Id-f_{1}$ satisfy the
following conditions 
\begin{equation*}
\left\Vert f_{1}\left( x\right) -f_{1}\left( x_{0}\right) \right\Vert
_{X}\leq \mu \left( \left\Vert x-x_{0}\right\Vert _{X}\right) ,\quad \
\forall x\in B_{r_{0}}^{X}\left( x_{0}\right) ,
\end{equation*}%
\begin{equation}
\left\langle f\left( x\right) -f\left( x_{0}\right) ,g\left( x-x_{0}\right)
\right\rangle \geq \nu \left( \left\Vert x-x_{0}\right\Vert _{X}\right)
\left\Vert x-x_{0}\right\Vert _{X},\quad \forall x\in B_{r_{0}}^{X}\left(
x_{0}\right) ,  \tag{2.4}
\end{equation}%
and almost each $\widetilde{x}\in intB_{r_{0}}^{X}\left( x_{0}\right) $
possesses a neighborhood $V_{\varepsilon }\left( \widetilde{x}\right) $, $%
\varepsilon \geq \varepsilon _{0}>0$ such that the following inequality%
\begin{equation*}
\left\Vert f\left( x_{2}\right) -f\left( x_{1}\right) \right\Vert _{X}\geq
\varphi \left( \left\Vert x_{2}-x_{1}\right\Vert _{X},\widetilde{x}%
,\varepsilon \right) ,
\end{equation*}%
holds for any $x_{1},x_{2}\in V_{\varepsilon }\left( \widetilde{x}\right)
\cap B_{r_{0}}^{X}\left( x_{0}\right) $, where $g:D\left( g\right) \subseteq
X\longrightarrow X^{\ast }$ such that $B_{r_{0}}^{X}\left( 0\right)
\subseteq D\left( g\right) $ and $g$ satisfies condition (ii), $\mu $ and $%
\nu $ are such functions as in Theorem 1, function $\varphi \left( \tau ,%
\widetilde{x},\varepsilon \right) $ has such a form as the right hand side
of inequality (2.3) (in particular, $g\equiv J:X\rightleftarrows X^{\ast }$,
i.e. $g$ be a duality mapping). Then the operator $f_{1}$ possesses a
fixed-point on the ball $B_{r_{0}}^{X}\left( x_{0}\right) $.
\end{theorem}

\begin{definition}
We call that an operator $f:D\left( f\right) \subseteq X\longrightarrow Y$
possesses the \textrm{P-property} iff any precompact subset $M$ of $Y$ from $%
\func{Im}\ f$ \ has a (general) subsequence $M_{0}\subset M$ such that there
exists a precompact subset $G$ of $X$ that satisfies the inclusions $%
f^{-1}\left( M_{0}\right) \subseteq G$ and $f\left( G\cap D\left( f\right)
\right) \supseteq M_{0}$.
\end{definition}

\begin{notation}
We can take the following condition instead of condition (iii) of Theorem 1: 
$f$ possesses the \textrm{P-property} on the ball $B_{r_{0}}^{X}\left(
0\right) $. It should be noted that an operator $f:D\left( f\right)
\subseteq X\longrightarrow Y$ possesses of the \textrm{P-property }if $%
f^{-1} $ is a lower or upper semi-continuous mapping.
\end{notation}

In the above results for the completeness of the image ($\func{Im}\ f$) of
the imposed operator $f,$ the condition \textit{(iii)} and \textrm{P-property%
} (and also the generalizations of the conditions \textit{(iii)}) are used%
\textit{.} But there are some other types of the complementary conditions on 
$f$ under which $\func{Im}\ f$ will be a closed subset. These types of
conditions are described in [1, 12, 14]. Therefore we do not conduct them
here again.

\section{General Result on Solvability of Perturbed Equation}

Now, we lead a solvability theorem for the perturbed nonlinear equation in
the Banach spaces, proved by using Theorem 1 and Corollary 1.

Let $X,Y$ be reflexive Banach spaces and $X^{\ast },Y^{\ast }$ be their dual
spaces, let $F:D\left( F\right) \subseteq X$ $\longrightarrow Y$ be a
nonlinear operator that has the representation $F\left( x\right) \equiv
F_{0}\left( x\right) +F_{1}\left( x\right) $ for any $x\in D\left( F\right) $%
, where $F_{i}:D\left( F_{i}\right) \subseteq X$ $\longrightarrow Y$ , $%
i=0,1 $ are some operators such that $D\left( F_{0}\right) \cap D\left(
F_{1}\right) \equiv G\subseteq X$ and $G\neq \varnothing $.

Consider the following equation 
\begin{equation}
F\left( x\right) \equiv F_{0}\left( x\right) +F_{1}\left( x\right) =y,\quad
y\in Y\ ,  \tag{3.1}
\end{equation}%
where $y$ is an element of $Y$.

Let $B_{r}^{X}\left( 0\right) \subseteq G\subseteq X$ be a closed ball, $r>0$
be a number. We set the following conditions

1) $F_{0}:B_{r}^{X}\left( 0\right) $ $\longrightarrow Y$ is the continuous
operator with its inverse operator $F_{0}^{-1}$, (as $F_{0}^{-1}:$ $D\left(
F_{0}^{-1}\right) \subseteq $ $Y$ $\longrightarrow $ $X$);

2) $F_{1}:B_{r}^{X}\left( 0\right) $ $\longrightarrow Y$ is a nonlinear
continuous operator;

3) there are such continuous functions $\mu _{i}:R_{+}^{1}\longrightarrow
R_{+}^{1}$ , $i=1,2$ and $\nu :R_{+}^{1}\longrightarrow R^{1}$ that the
following inequalities 
\begin{equation*}
\left\Vert F_{0}\left( x\right) \right\Vert _{Y}\leq \mu _{1}\left(
\left\Vert x\right\Vert _{X}\right) \ \&\ \left\Vert F_{1}\left( x\right)
\right\Vert _{Y}\leq \mu _{2}\left( \left\Vert x\right\Vert _{X}\right) ,
\end{equation*}%
\begin{equation*}
\left\langle F_{0}\left( x\right) +F_{1}\left( x\right) ,g\left( x\right)
\right\rangle \geq c\left\langle F_{0}\left( x\right) ,g\left( x\right)
\right\rangle \geq \nu \left( \left\Vert x\right\Vert _{X}\right) \left\Vert
x\right\Vert _{X},
\end{equation*}%
hold for any $x\in B_{r}^{X}\left( 0\right) $, moreover $\nu \left( r\right)
\geq $ $\delta _{0}$ holds for some number $\delta _{0}>0$, where the
mapping $g:B_{r}^{X}\left( 0\right) \subseteq D\left( g\right) \subseteq
X\longrightarrow Y^{\ast }$ fulfills the conditions of Theorem 1, where $c>0$
is a constant;

4) almost each $\widetilde{x}\in intB_{r}^{X}\left( 0\right) $ possesses a
neighborhood $B_{\varepsilon }^{X}\left( \widetilde{x}\right) $, $%
\varepsilon \geq \varepsilon _{0}>0$, such that the following inequality 
\begin{equation*}
\left\Vert F\left( x_{1}\right) -F\left( x_{2}\right) \right\Vert _{Y}\geq
c_{1}\left( \widetilde{x},\varepsilon \right) \left\Vert F_{0}\left(
x_{1}\right) -F_{0}\left( x_{2}\right) \right\Vert _{Y},
\end{equation*}%
holds for any $x_{1},x_{2}\in B_{\varepsilon }^{X}\left( \widetilde{x}%
\right) $ and some number $\varepsilon _{0}>0$, where $c_{1}\left(
\left\Vert \widetilde{x}\right\Vert _{X},\varepsilon \right) >0$ is bounded
for each $\widetilde{x}\in intB_{r}^{X}\left( 0\right) $;

or

4') operator $F:B_{r}^{X}\left( 0\right) \subset X$ $\longrightarrow Y$
possesses \textrm{P-property}, except for the existence of the inverse
operator $F_{0}^{-1}$ in condition 1.

Then the following statement is true by Theorem 1.

\begin{theorem}
Let conditions 1, 2, 3, 4 (or 1, 2, 3, 4') be fulfilled then equation (3.1)
has a solution in the ball $B_{r}^{X}\left( 0\right) $ for any $y\in Y$ that
fulfills the following inequality 
\begin{equation*}
\left\langle y,g\left( x\right) \right\rangle \leq \nu \left( \left\Vert
x\right\Vert _{X}\right) \left\Vert x\right\Vert _{X},\quad \forall x\in
S_{r}^{X}\left( 0\right) .
\end{equation*}
\end{theorem}

\section{Fully Nonlinear Equations of Second Order}

Now, we study some nonlinear BVP with using the general results. Let $\Omega
\subset R^{n}$ ($n\geq 1$) be an open bounded domain with sufficiently
smooth boundary $\partial \Omega $. Consider the following problem 
\begin{equation}
f\left( u\right) \equiv -\Delta u+F\left( x,u,Du,\Delta u\right) =0,\quad
x\in \Omega ,\quad u\left\vert ~_{\partial \Omega }\right. =0,  \tag{4.1}
\end{equation}%
where $F\left( x,\xi ,\eta ,\zeta \right) $ is a Caratheodory function on $%
\Omega \times R^{2}\times R^{n}$ as $F:\Omega \times R^{2}\times
R^{n}\longrightarrow R^{1}$, $D\equiv \left( D_{1},D_{2},\ldots
,D_{n}\right) $, $\Delta \equiv \overset{n}{\underset{i=1}{\sum }}D_{i}^{2}$
(is Laplacian), $D_{i}\equiv \frac{\partial }{\partial x_{i}}$.

Let the following conditions be fulfilled

(\textit{i}) there are Caratheodory functions $F_{0}\left( x,\xi \right) $ , 
$F_{1}\left( x,\xi ,\eta \right) $, $F_{2}\left( x,\xi ,\eta ,\zeta \right) $
: $F_{0},F_{1},F_{2}:\Omega \times R^{n+2}\longrightarrow R^{1}$ such that $%
F\left( x,\xi ,\eta ,\zeta \right) =F_{0}\left( x,\xi \right) +F_{1}\left(
x,\xi ,\eta \right) +F_{2}\left( x,\xi ,\eta ,\zeta \right) $ for any $%
\left( x,\xi ,\eta ,\zeta \right) \in \Omega \times R^{1}\times R^{n}\times
R^{1}$, moreover

(\textit{a}) there exist a Caratheodory function $a_{1}\left( x,\xi \right) $
and numbers $m_{0},\nu \geq 0$, $\mu ,M>0$, $\mu +2>2\nu $ such that 
\begin{equation*}
F_{0}\left( x,\xi \right) \equiv M\ \left\vert \xi \right\vert ^{\mu }\ \xi
+a_{1}\left( x,\xi \right) ,\quad
\end{equation*}%
\begin{equation}
\left\vert a_{1}\left( x,\xi \right) \right\vert \leq m_{0}\ \left\vert \xi
\right\vert ^{\nu }+\psi \left( x\right) ,\ \psi \in L_{p}\left( \Omega
\right) ,\ p>2\ ,  \tag{4.2}
\end{equation}%
hold for a.e. $x\in \Omega $ and any $\xi \in R^{1}$, and

(\textit{b}) there exist a number $\rho :$ $2\geq \rho \geq 0$ and a
nonnegative Caratheodory function $m_{1}\left( x,\xi ,\eta \right) \geq 0$
such that \ 
\begin{equation}
\left\vert F_{1}\left( x,\xi ,\eta \right) \right\vert \leq m_{1}\left(
x,\xi ,\eta \right) ~\left\vert \eta \right\vert ^{\rho }+k\left( x\right)
,\ \forall \left( x,\xi ,\eta \right) \in \Omega \times R^{1}\times R^{n}, 
\tag{4.3}
\end{equation}%
holds, where $m_{1}\left( x,\xi ,\eta \right) \leq M_{1}\left( x\right) $,
and $2\left( \widehat{C}\left( \mu ,\rho ,n\right) \left\Vert
M_{1}\right\Vert _{\infty }^{2}\right) ^{2}$ $\leq M$, $k\in L^{p_{1}}\left(
\Omega \right) ,\quad p_{1}>2$, where $\widehat{C}\left( \mu ,\rho ,n\right) 
$ is the coefficient of Gagliardo-Nirenberg-Sobolev (G-N-S) inequality (see,
[16]).\footnote{%
\begin{equation*}
\left\Vert D^{\beta }v\right\Vert _{p_{0}}\leq C\left(
p_{0},p_{1},p_{2},l,s,n\right) \times 
\end{equation*}%
\par
\begin{equation*}
\left[ \underset{\left\vert \alpha \right\vert \leq l}{\sum }\left\Vert
D^{\alpha }v\right\Vert _{p_{1}}\right] ^{\theta }\times \left\Vert
v\right\Vert _{p_{2}}^{1-\theta },
\end{equation*}%
\par
where $\theta \in \left[ \frac{\left\vert \beta \right\vert }{l},1\right] $
and satisfy the equation%
\begin{equation*}
\frac{n}{p_{0}}-s=\theta \left( \frac{n}{p_{1}}-l\right) +\left( 1-\theta
\right) \frac{n}{p_{2}},
\end{equation*}%
the constant $C\left( p_{0},p_{1},p_{2},l,s,n\right) $ is independent of $%
v\left( x\right) $; $p_{0},p_{1},p_{2}\geq 1$, $0\leq s<l$, $\left\vert
\beta \right\vert =\underset{i=1}{\overset{n}{\sum }}\beta _{i}=s$, $\beta
_{i}\geq 0$ be some numbers.}

c) there exist Carathedory functions $c\left( x,\xi ,\eta ,\zeta \right) $, $%
\widetilde{F}_{2}\left( x,\xi ,\eta \right) $ and a continuous function $%
k\left( \zeta \right) $ such that the following inequalities 
\begin{equation}
\left\vert F_{2}\left( x,\xi ,\eta ,\zeta \right) -F_{2}\left( x,\xi ,\eta
,\zeta _{1}\right) \right\vert \leq c\left( x,\xi ,\eta ,\zeta ,\zeta
_{1}\right) \ \left\vert \zeta -\zeta _{1}\right\vert ,  \tag{4.4}
\end{equation}%
\begin{equation}
\left\vert F_{2}\left( x,\xi ,\eta ,\zeta \right) -F_{2}\left( x,\xi
_{1},\eta _{1},\zeta \right) \right\vert \leq k\left( \zeta \right)
\left\vert \widetilde{F}_{2}\left( x,\xi ,\eta \right) -\widetilde{F}%
_{2}\left( x,\xi _{1},\eta _{1}\right) \right\vert ,  \tag{4.5}
\end{equation}%
hold for a.e. $x\in \Omega $ and any $\left( \xi ,\eta \right) ,\left( \xi
_{1},\eta _{1}\right) \in R^{n+1}$, $\forall \zeta ,\ \zeta _{1}\in R^{1}$,
and $F_{2}\left( x,\xi ,\eta ,0\right) =0$, moreover there exists a function 
$\psi _{1}\in L^{\infty }\left( \Omega \right) $ such that $\psi _{1}\left(
x\right) \geq 0$, $c\left( x,\xi ,\eta ,\zeta \right) \leq \psi _{1}\left(
x\right) $ hold for a.e. $x\in \Omega $ and any $\left( \xi ,\eta ,\zeta
\right) \in R^{N}$, here $\left\Vert \psi _{1}\right\Vert _{L^{\infty
}\left( \Omega \right) }\equiv \left\Vert \psi _{1}\right\Vert _{\infty
}\leq 4^{-1}$.

Assume the following denotations: $\left\Vert u\right\Vert _{L_{p}\left(
\Omega \right) }\equiv \left\Vert u\right\Vert _{p}$ for any $p\in \left[
1,\infty \right] $, and $\left\Vert u\right\Vert _{W^{l,p}\left( \Omega
\right) }\equiv \left\Vert u\right\Vert _{l,p}$, for $u\in W^{l,p}\left(
\Omega \right) $, $l\geq 1$.

So, we consider the operator $f:W^{2,2}\left( \Omega \right) \cap
W_{0}^{1,2}\left( \Omega \right) \longrightarrow L^{2}\left( \Omega \right) $%
, which is generated by the imposed problem.

\begin{theorem}
Let conditions (i), (a), (b), (c) be fulfilled and parameters $\mu $ and $%
\rho $ satisfy the following relations 
\begin{equation*}
1<\mu \leq \frac{4}{n-4}\ \text{if}\ n\geq 5\ \&\ 1<\mu <\infty \text{ }\ 
\text{if }\ n=2,3,4,
\end{equation*}%
\begin{equation*}
\rho \geq \frac{\left( \mu +2\right) n}{2\left( n+\mu \right) }\ \&\ \rho
\leq 1+\min \left\{ \frac{\mu +2}{n+\mu };\frac{2}{n-2};\frac{\mu }{\mu +2}%
\right\} .
\end{equation*}

Then problem (4.1) is solvable in $W^{2,2}\left( \Omega \right) \cap
W_{0}^{1,2}\left( \Omega \right) $.
\end{theorem}

\begin{proof}
For the proof, it is enough to show that the operator $f:W^{2,2}\left(
\Omega \right) \cap W_{0}^{1,2}\left( \Omega \right) \longrightarrow
L^{2}\left( \Omega \right) $ fulfills all conditions of Theorem 1 (or
Theorem 3). For this, we will estimate the following dual form 
\begin{equation*}
\left\langle f\left( u\right) ,g\left( u\right) \right\rangle =\left\langle
-\Delta u+F\left( x,u,Du,\Delta u\right) ~,~-\Delta u\right\rangle ,
\end{equation*}%
where the operator $g$ is determined in the form $g\equiv -\Delta
:W^{2,2}\left( \Omega \right) \cap W_{0}^{1,2}\left( \Omega \right)
\longrightarrow L^{2}\left( \Omega \right) $, then we have 
\begin{equation*}
\left\langle -\Delta u+F\left( x,u,Du,\Delta u\right) ,-\Delta
u\right\rangle =\left\Vert \Delta u\right\Vert _{2}^{2}+\left\langle F\left(
x,u,Du,\Delta u\right) ,-\Delta u\right\rangle ,
\end{equation*}%
consequently we need\ to estimate the second term of this equation, i.e. the
dual form: $\left\langle F\left( x,u,Du,\Delta u\right) ,-\Delta
u\right\rangle $.

Thus using conditions (i), (a), (b), (c) we obtain 
\begin{equation*}
\left\langle F\left( x,u,Du,\Delta u\right) ,-\Delta u\right\rangle
=-\left\langle F_{0}\left( x,u\right) ,\Delta u\right\rangle -\left\langle
F_{1}\left( x,u,Du\right) ,\Delta u\right\rangle -
\end{equation*}%
\begin{equation*}
\left\langle F_{2}\left( x,u,Du,\Delta u\right) ,\Delta u\right\rangle
=-\left\langle M\ \left\vert u\right\vert ^{\mu }\ u+a_{1}\left( x,u\right)
,\Delta u\right\rangle -
\end{equation*}%
\begin{equation*}
\left\langle F_{1}\left( x,u,Du\right) ,\Delta u\right\rangle -\left\langle
F_{2}\left( x,u,Du,\Delta u\right) ,\Delta u\right\rangle \geq \left\langle
M\left( \mu +1\right) \ \left\vert u\right\vert ^{\mu }\ \nabla u,\nabla
u\right\rangle -
\end{equation*}%
\begin{equation*}
\left\Vert m_{0}\left\vert u\right\vert ^{\nu }+\psi \right\Vert
_{2}\left\Vert \Delta u\right\Vert _{2}-\left\Vert m_{1}\left( x,u,Du\right)
~\left\vert \nabla u\right\vert ^{\rho }+k\left( x\right) \right\Vert
_{2}\left\Vert \Delta u\right\Vert _{2}-
\end{equation*}%
\begin{equation*}
\left\Vert c\left( x,u,Du,\Delta u\right) \ \left\vert \Delta u\right\vert
\right\Vert _{2}\left\Vert \Delta u\right\Vert _{2}\geq M\left( \mu
+1\right) \left\Vert \left\vert u\right\vert ^{\frac{\mu }{2}}\ \nabla
u\right\Vert _{2}^{2}-
\end{equation*}%
\begin{equation*}
\varepsilon _{1}\left\Vert \left\vert u\right\vert ^{\frac{\mu +2}{2}%
}\right\Vert _{2}-\left\Vert M_{1}\right\Vert _{\infty }^{2}\left\Vert
\left\vert \nabla u\right\vert ^{\rho }\right\Vert _{2}^{2}-\left(
\varepsilon +4^{-1}+\left\Vert \psi _{1}\right\Vert _{\infty }\right)
\left\Vert \Delta u\right\Vert _{2}^{2}-
\end{equation*}%
\begin{equation}
C\left( \varepsilon ,\varepsilon _{1},m_{0},\left\Vert \psi \right\Vert
_{p},\left\Vert k\right\Vert _{p}\right) ,\quad \varepsilon ,\varepsilon
_{1}\in \left( 0,1\right) .  \tag{4.6}
\end{equation}

Hence we need to estimate the term $\left\Vert \left\vert \nabla
u\right\vert ^{\rho }\right\Vert _{2}^{2}$ by using $\left\Vert \left\vert
u\right\vert ^{\frac{\mu }{2}}\ u\right\Vert _{1,2}^{2}$ and $\left\Vert
u\right\Vert _{2,2}^{2}$ for which we will use the known inequality (G-N-S).
Using that we get 
\begin{equation*}
\left\langle F\left( x,u,Du,\Delta u\right) ,-\Delta u\right\rangle \geq
\left\langle M\left( \varepsilon _{1}\right) \left( \mu +1\right) \
\left\vert u\right\vert ^{\mu }\ \nabla u,\nabla u\right\rangle -\left(
\varepsilon +2^{-1}\right) \left\Vert \Delta u\right\Vert _{2}^{2}-
\end{equation*}%
\begin{equation*}
\widehat{C}\left( \mu ,\rho ,n\right) \left\Vert M_{1}\right\Vert _{\infty
}^{2}\left\Vert \Delta u\right\Vert _{2}^{\theta \rho }\left\Vert
u\right\Vert _{\frac{\left( \mu +2\right) n}{n-2}}^{\left( 1-\theta \right)
\rho }-C\left( \varepsilon ,\varepsilon _{1},m_{0},\left\Vert \psi
\right\Vert _{p},\left\Vert k\right\Vert _{p}\right) ,
\end{equation*}%
where $\theta =\left[ 2\rho \left( n+\mu \right) -\left( \mu +2\right) n%
\right] \cdot \left[ 4\rho \left( \mu +1\right) -\rho \mu n\right] ^{-1}$
for the considered case, and here $\widehat{C}\left( \mu ,\rho ,n\right) $
is the coefficient of the inequality G-N-S.

From here we obtain 
\begin{equation*}
\left\langle F\left( x,u,Du,\Delta u\right) ,-\Delta u\right\rangle \geq
\left\langle M_{\varepsilon _{1}}\left( \mu +1\right) \ \left\vert
u\right\vert ^{\mu }\ \nabla u,\nabla u\right\rangle -\left( \varepsilon +%
\frac{3}{4}\right) \left\Vert \Delta u\right\Vert _{2}^{2}-
\end{equation*}%
\begin{equation}
\left( \widehat{C}\left( \mu ,\rho ,n\right) \left\Vert M_{1}\right\Vert
_{\infty }^{2}\right) ^{2}\left\Vert u\right\Vert _{\frac{\left( \mu
+2\right) n}{n-2}}^{\mu +2}-C\left( \varepsilon ,\varepsilon
_{1},m_{0},\left\Vert \psi \right\Vert _{p},\left\Vert k\right\Vert
_{p}\right) ,  \tag{4.7}
\end{equation}%
where $\left( \widehat{C}\left( \mu ,\rho ,n\right) \left\Vert
M_{1}\right\Vert _{\infty }^{2}\right) ^{2}<M$ by the condition \textit{(b)}.

Now if we take into account (4.6) and (4.7), we get 
\begin{equation*}
\left\langle f\left( u\right) ,g\left( u\right) \right\rangle =\left\langle
-\Delta u+F\left( x,u,Du,\Delta u\right) ,-\Delta u\right\rangle \geq \left( 
\frac{1}{4}-\varepsilon \right) \left\Vert \Delta u\right\Vert _{2}^{2}+
\end{equation*}%
\begin{equation}
\widetilde{M}\ \left( \mu +1\right) \left\Vert \ \left\vert u\right\vert ^{%
\frac{\mu }{2}}\ \nabla u\right\Vert _{2}^{2}-C\left( \varepsilon
,\varepsilon _{1},m_{0},\left\Vert \psi \right\Vert _{p},\left\Vert
k\right\Vert _{p}\right) .  \tag{4.8}
\end{equation}

So, condition 3 of Theorem 3 is fulfilled since the operator $%
f:W^{2,2}\left( \Omega \right) \cap W_{0}^{1,2}\left( \Omega \right)
\longrightarrow L^{2}\left( \Omega \right) $ is bounded, that can be seen
easily from its expression and the conditions of this theorem.

Thus it follows that problem (4.1) is densely solvable in $L^{2}\left(
\Omega \right) $. Consequently, it remains to show that image $f\left(
W^{2,2}\left( \Omega \right) \cap W_{0}^{1,2}\left( \Omega \right) \right) $
is closed in the space $L^{2}\left( \Omega \right) $.

Let $h_{0}\in L^{2}\left( \Omega \right) $ then there is a sequence $\left\{
h_{m}\right\} _{m=1}^{\infty }\subset \func{Im}\ f\subseteq $ that $%
L^{2}\left( \Omega \right) $ converges to the given $h_{0}$ in $L^{2}\left(
\Omega \right) $, as $cl\ \mathrm{\func{Im}}\ f$ $\equiv L^{2}\left( \Omega
\right) $. For any $h_{m}$ we have the subset $f^{-1}\left( h_{m}\right) $
and since $\left\{ h_{m}\right\} $ is a bounded subset in $L^{2}\left(
\Omega \right) $ then there is a bounded subset $G$ of $W^{2,2}\left( \Omega
\right) \cap W_{0}^{1,2}\left( \Omega \right) $ such that $G\cap
f^{-1}\left( \left\{ h_{m}\right\} _{m=1}^{\infty }\right) \neq \varnothing $%
, in addition $G\cap f^{-1}\left( h_{m}\right) \neq \varnothing $ for any $%
h_{m}$. Then we can choose a sequence $\left\{ u_{m}\right\} \subset G$ such
that $f\left( u_{m}\right) =h_{m}$ which belongs to the bounded subset $G$.
From here, by using the reflexivity of the space $W^{2,2}\left( \Omega
\right) \cap W_{0}^{1,2}\left( \Omega \right) $ we can select a subsequence $%
\left\{ u_{m_{k}}\right\} _{k=1}^{\infty }\subseteq \left\{ u_{m}\right\}
_{m=1}^{\infty }$ that is a weakly convergent sequence in $W^{2,2}\left(
\Omega \right) \cap W_{0}^{1,2}\left( \Omega \right) $, i.e. there is an
element $u_{0}$ such that $u_{m_{k}}\rightharpoonup u_{0}$ in $W^{2,2}\left(
\Omega \right) \cap W_{0}^{1,2}\left( \Omega \right) $ (may be after the
choice of a subsequence of $\left\{ u_{m_{k}}\right\} _{k=1}^{\infty }$),
and consequently $u_{m_{k}}\longrightarrow u_{0}$ in $W^{1,p}\left( \Omega
\right) $, $1\leq p<2^{\ast }$.

Thus, we get $F_{0}\left( x,u_{m_{k}}\right) \longrightarrow F_{0}\left(
x,u_{0}\right) $, $F_{1}\left( x,u_{m_{k}},Du_{m_{k}}\right) \longrightarrow
F_{1}\left( x,u_{0},Du_{0}\right) $ in $L^{2}\left( \Omega \right) $ by the
conditions of Theorem 4 that $F_{i}:W^{2,2}\left( \Omega \right) \cap
W_{0}^{1,2}\left( \Omega \right) \longrightarrow L^{2}\left( \Omega \right) $%
, $i=0,1$, are continuous operators.

On the other hand for any $\varepsilon >0$ there exist $m_{k}$, $m_{l}\geq
m_{k}\left( \varepsilon \right) \geq 1$ such that the inequality 
\begin{equation*}
\varepsilon >\left\Vert h_{m_{k}}-h_{m_{l}}\right\Vert _{2}\equiv \left\Vert
f\left( u_{m_{k}}\right) -f\left( u_{m_{l}}\right) \right\Vert _{2}\geq
\end{equation*}%
\begin{equation*}
\frac{3}{4}\left\Vert \Delta \left( u_{m_{k}}-u_{m_{l}}\right) \right\Vert
_{2}+\widehat{M}\left\Vert u_{m_{k}}-u_{m_{l}}\right\Vert _{2\left( \mu
+1\right) }^{\mu +1}-
\end{equation*}%
\begin{equation*}
\left\Vert F_{0}\left( x,u_{m_{k}}\right) -F_{0}\left( x,u_{m_{l}}\right)
\right\Vert _{2}-\left\Vert F_{1}\left( x,u_{m_{k}},Du_{m_{k}}\right)
-F_{1}\left( x,u_{m_{l}},Du_{m_{l}}\right) \right\Vert _{2}-
\end{equation*}%
\begin{equation}
k\left( \left\Vert \Delta u_{m_{k}}\right\Vert _{2}\right) \left\Vert 
\widetilde{F}_{2}\left( x,u_{m_{k}},Du_{m_{k}}\right) -\widetilde{F}%
_{2}\left( x,u_{m_{l}},Du_{m_{l}}\right) \right\Vert _{2},  \tag{4.9}
\end{equation}%
holds, where $u_{m_{k}}\longrightarrow u_{0}$ in $W^{1,p}\left( \Omega
\right) $, $1\leq p<2^{\ast }$ and $u_{m_{k}}\rightharpoonup u_{0}$ in $%
W^{2,2}\left( \Omega \right) $. Hence we obtain that the last terms of (4.9)
converge to zero under $m_{k}\nearrow \infty $, then we get 
\begin{equation*}
\left\Vert \Delta \left( u_{m_{k}}-u_{m_{l}}\right) \right\Vert _{2}\searrow
0\quad \text{if \ }m_{k},m_{l}\nearrow \infty .
\end{equation*}

Consequently $\Delta u_{m_{k}}\longrightarrow \Delta u_{0}$ in $L^{2}\left(
\Omega \right) $, $F_{2}\left( x,u_{m_{k}},Du_{m_{k}},\Delta
u_{m_{k}}\right) \longrightarrow F_{2}\left( x,u_{0},Du_{0},\Delta
u_{0}\right) $ in $L^{2}\left( \Omega \right) $ and from the equality 
\begin{equation*}
\left\langle -\Delta u_{m_{k}}+F\left( x,u_{m_{k}},Du_{m_{k}},\Delta
u_{m_{k}}\right) ,v\right\rangle =\left\langle -\Delta
u_{m_{k}},v\right\rangle +\left\langle F_{0}\left( x,u_{m_{k}}\right)
,v\right\rangle -
\end{equation*}%
\begin{equation*}
\left\langle F_{1}\left( x,u_{m_{k}},Du_{m_{k}}\right) ,v\right\rangle
-\left\langle F_{2}\left( x,u_{m_{k}},Du_{m_{k}},\Delta u_{m_{k}}\right)
,v\right\rangle =\left\langle h_{m_{k}},v\right\rangle ,\quad \forall v\in
L_{2}\left( \Omega \right) \ \&\forall k\geq 1,
\end{equation*}%
we obtain that 
\begin{equation*}
\left\langle -\Delta u_{0}+F\left( x,u_{0},Du_{0},\Delta u_{0}\right)
,v\right\rangle =\left\langle h_{0},v\right\rangle ,\quad \forall v\in
L_{2}\left( \Omega \right) .
\end{equation*}%
Hence it follows that $h_{0}\in \mathrm{\func{Im}}\ f$, i.e. $\mathrm{\func{%
Im}}\ f\equiv f\left( W^{2,2}\left( \Omega \right) \cap W_{0}^{1,2}\left(
\Omega \right) \right) \equiv L^{2}\left( \Omega \right) $.
\end{proof}

\begin{remark}
The result of this theorem shows that we can consider the following problem%
\begin{equation*}
-\Delta u+M\ \left\vert u\right\vert ^{\mu }\ u=-a_{1}\left( x,u\right)
-F_{1}\left( x,u,Du\right) -F_{2}\left( x,u,Du,\Delta u\right) 
\end{equation*}%
let the operators $G_{0}:W^{2,2}\left( \Omega \right) \cap \overset{0}{W}%
~^{1,2}\left( \Omega \right) \longrightarrow L^{2}\left( \Omega \right) $
and $G_{1}:W^{2,2}\left( \Omega \right) \cap W_{0}^{1,2}\left( \Omega
\right) \longrightarrow L^{2}\left( \Omega \right) $ be defined by the
expressions 
\begin{equation*}
G_{0}\left( u\right) \equiv -\Delta u+M\ \left\vert u\right\vert ^{\mu }\
u,\quad G_{1}\left( u\right) \equiv -a_{1}\left( x,u\right) -F_{1}\left(
x,u,Du\right) -F_{2}\left( x,u,Du,\Delta u\right) 
\end{equation*}%
respectively, then existence of $G_{0}^{-1}:L^{2}\left( \Omega \right)
\longrightarrow W^{2,2}\left( \Omega \right) \cap W_{0}^{1,2}\left( \Omega
\right) $ is known and $G_{0}^{-1}$ is a bounded continuous operator. Now we
determine the operator $G\left( u\right) \equiv \left( G_{0}^{-1}\circ
G_{1}\right) \ \left( u\right) $ that acts from $W^{2,2}\left( \Omega
\right) \cap W_{0}^{1,2}\left( \Omega \right) $ to $W^{2,2}\left( \Omega
\right) \cap W_{0}^{1,2}\left( \Omega \right) $ and is bounded continuous
operator under the conditions of Theorem 4. Hence we obtain that the
operator $G$ possesses a fixed point that allows us to investigate the
problem on the existence of the eigenvalue of the operator $G_{0}$ relative
to the operator $G_{1}$, i.e. study of the problem $G_{0}\left( u\right)
=\lambda G_{1}\left( u\right) $. But in this case we can only conclude that $%
\lambda $ will be dependent on $u$.
\end{remark}

\begin{remark}
From the proof of this theorem it follows that the result of such type is
true and in the case of the operator $F\left( x,u,Du,\Delta u\right) $ in
the problem (4.1) is independent of $\Delta u$, i.e. it has the
representation $F\left( x,\xi ,\eta ,\zeta \right) \equiv F\left( x,\xi
,\eta \right) \equiv F_{0}\left( x,\xi \right) +F_{1}\left( x,\xi ,\eta
\right) $ for $\left( x,\xi ,\eta \right) \in \Omega \times \Re ^{n+1}.$
\end{remark}

\section{Nonlinear Equation with $p-$Laplacian}

On the open bounded domain $\Omega \subset R^{n}$ with sufficiently smooth
boundary $\partial \Omega $ consider the following problem 
\begin{equation}
f\left( u\right) \equiv -\nabla \left( \left\vert \nabla u\right\vert
^{p-2}\nabla u\right) +G\left( x,u,Du,\lambda ,\mu \right) =h\left( x\right)
,\quad x\in \Omega \subset R^{n},  \tag{5.1}
\end{equation}%
\begin{equation}
u\left\vert \ _{\partial \Omega }\right. =0,\quad n\geq 1,\ \Omega \in Lip,\
\ h\in W^{-1,q}\left( \Omega \right) .  \tag{5.2}
\end{equation}

Assume that 
\begin{equation}
G\left( x,\xi ,\eta ,\lambda \right) =\mu G_{0}\left( x,\xi \right) +\lambda
G_{1}\left( x,\xi ,\eta \right) ,  \tag{5.3}
\end{equation}%
holds for a.e. $x\in \Omega $ and any $\left( \xi ,\eta \right) \in R^{n+1}$%
, where $G_{1}\left( x,\xi ,\eta \right) $ and $G_{0}\left( x,\xi \right) $
are some Caratheodory functions, $\lambda \in R$, $\mu \geq 0$ are some
parameters.

\textbf{5.1. Dense solvability.} Let the following conditions 
\begin{equation}
G_{0}\left( x,\xi \right) \cdot \xi \geq a_{0}\left( x\right) \left\vert \xi
\right\vert ^{p_{0}}-a_{1}\left( x\right) ,\quad a_{0}\left( x\right) \geq
A_{0}>0;  \tag{5.4}
\end{equation}%
\begin{equation*}
\left\vert G_{0}\left( x,\xi \right) \right\vert \leq \widetilde{a}%
_{0}\left( x\right) \left\vert \xi \right\vert ^{p_{0}-1}+\widetilde{a}%
_{1}\left( x\right) ,\quad \widetilde{a}_{0}\left( x\right) ,\widetilde{a}%
_{1}\left( x\right) \geq 0,
\end{equation*}%
\begin{equation}
\left\vert G_{1}\left( x,\xi ,\eta \right) \right\vert \leq b_{0}\left(
x\right) \left\vert \eta \right\vert ^{p_{1}}+b_{1}\left( x\right)
\left\vert \xi \right\vert ^{p_{2}}+b_{2}\left( x\right) ,\ b_{j}\left(
x\right) \geq 0,\ j=0,1,2,  \tag{5.5}
\end{equation}%
hold for a.e. $x\in \Omega $ and any $\left( \xi ,\eta \right) \in R^{n+1}$
where $p_{0},p_{1},p_{2}\geq 0$, $p>1$ are some numbers, $a_{k}\left(
x\right) $, $\widetilde{a}_{k}\left( x\right) $ and $b_{j}\left( x\right) $
are some functions, $k=0,1$ and $j=0,1,2$.

Here we study the solvability of problem (5.1)-(5.2) in the generalized
sense, i.e. a function $u\in W_{0}^{1,p}\left( \Omega \right) $ is called a
solution of the problem (5.1), (5.2) if $u$\ satisfies the equation 
\begin{equation*}
\left\langle f\left( u\right) ,v\right\rangle =\left\langle h,v\right\rangle
,\quad v\in W_{0}^{1,p}\left( \Omega \right) ,
\end{equation*}%
for any $v\in W_{0}^{1,p}\left( \Omega \right) $.

\begin{theorem}
Let conditions (5.3) - (5.5) be fulfilled and $p_{0}-1\geq p_{2}\geq 0$, $p>1
$, $p>p_{1}\geq 0$. Moreover, let $a_{k}\left( x\right) $ , $\widetilde{a}%
_{k}\left( x\right) $, $b_{j}\left( x\right) $ be such functions that $a_{0},%
\widetilde{a}_{0}\in L^{\infty }\left( \Omega \right) $, $a_{1},\widetilde{a}%
_{1}\in L^{q}\left( \Omega \right) $ and $b_{0},b_{1}\in L^{\infty }\left(
\Omega \right) $ , $b_{2}\in L^{q}\left( \Omega \right) $. If $p_{1}$, $p_{0}
$ satisfy the inequalities $p_{1}\leq p-\frac{p}{p_{0}}$, $p_{0}\leq p^{\ast
}\equiv \frac{pn}{n-p}$, then there exist a subset $\mathcal{M}\subseteq $ $%
W^{-1,q}\left( \Omega \right) $ and some numbers $\mu _{0}>0$, $C_{0}>0$
such that $\overline{\mathcal{M}}^{W^{-1,q}}\equiv W^{-1,q}\left( \Omega
\right) $, $q=\frac{p}{p-1}\equiv p^{\prime }$, and $\mu \geq \mu _{0}>0$ , $%
\lambda :\left\vert \lambda \right\vert \leq C_{0}$, problem (5.1)-(5.2) is
solvable in $W_{0}^{1,p}\left( \Omega \right) $ for any $h\in \mathcal{M}$;
moreover if $\widetilde{p}=p_{0}$ or $p_{2}+1=p_{0}$ then $C_{0}\equiv
C_{0}\left( A_{0},b_{0},b_{1},\mu \right) $ is\ sufficiently small number. \ 
\end{theorem}

\begin{proof}
Let $u\in W_{0}^{1,p}\left( \Omega \right) \cap L^{p_{0}}\left( \Omega
\right) $ and consider the dual form 
\begin{equation*}
\left\langle f\left( u\right) ,u\right\rangle \equiv \left\Vert \nabla
u\right\Vert _{p}^{p}+\left\langle G\left( x,u,Du,\lambda \right)
,u\right\rangle =
\end{equation*}%
\begin{equation*}
\left\Vert \nabla u\right\Vert _{p}^{p}+\left\langle \mu G_{0}\left(
x,u\right) ,u\right\rangle +\lambda \left\langle G_{1}\left( x,u,Du\right)
,u\right\rangle ,
\end{equation*}%
then by using conditions (5.4) and (5.5), we get 
\begin{equation*}
\left\langle f\left( u\right) ,u\right\rangle \geq \left\Vert \nabla
u\right\Vert _{p}^{p}+\mu \ \left\langle a_{0}\left( x\right) \left\vert
u\right\vert ^{p_{0}-2}u,u\right\rangle -\mu \left\Vert \ a_{1}\right\Vert
_{1}-
\end{equation*}%
\begin{equation*}
\left\vert \lambda \right\vert \left\langle b_{0}\left( x\right) \left\vert
\nabla u\right\vert ^{p_{1}},\left\vert u\right\vert \right\rangle
-\left\vert \lambda \right\vert \left\langle b_{1}\left( x\right) \left\vert
u\right\vert ^{p_{2}},\left\vert u\right\vert \right\rangle -\left\vert
\lambda \right\vert \left\langle b_{2}\left( x\right) ,\left\vert
u\right\vert \right\rangle ,
\end{equation*}%
or 
\begin{equation*}
\left\langle f\left( u\right) ,u\right\rangle \geq \left\Vert \nabla
u\right\Vert _{p}^{p}+\mu A_{0}\left\Vert u\right\Vert _{p_{0}}^{p_{0}}-\mu
\left\Vert a_{1}\right\Vert _{1}-\left\vert \lambda \right\vert \underset{%
\Omega }{\int }b_{0}\left( x\right) \left\vert \nabla u\right\vert
^{p_{1}}\left\vert u\right\vert dx-
\end{equation*}%
\begin{equation}
\left\vert \lambda \right\vert \underset{\Omega }{\int }b_{1}\left( x\right)
\left\vert u\right\vert ^{p_{2}+1}\ dx-\left\vert \lambda \right\vert 
\underset{\Omega }{\int }b_{2}\left( x\right) \left\vert u\right\vert \ dx. 
\tag{5.6}
\end{equation}

Since $b_{j}\in L^{\infty }\left( \Omega \right) $, $\ j=0,1$, it is enough
to estimate first integral of the right side of inequality (5.6). For this,
we have 
\begin{equation*}
\left\vert \lambda \right\vert \underset{\Omega }{\int }b_{0}\left( x\right)
\left\vert \nabla u\right\vert ^{p_{1}}\left\vert u\right\vert dx\leq
\left\vert \lambda \right\vert \left\Vert b_{0}\right\Vert _{\infty }\left[
\varepsilon \left\Vert \left\vert \nabla u\right\vert \right\Vert
_{p}^{p}+c\left( \varepsilon \right) \left\Vert u\right\Vert _{\widetilde{p}%
}^{\widetilde{p}}\right] .
\end{equation*}

By using the last inequality in (5.6) and taking into account the condition
on $p_{1},$ we obtain 
\begin{equation*}
\left\langle f\left( u\right) ,u\right\rangle \geq \left( 1-\varepsilon
\left\vert \lambda \right\vert \left\Vert b_{0}\right\Vert _{\infty }\right)
\left\Vert \left\vert \nabla u\right\vert \right\Vert _{p}^{p}+\left( \mu
A_{0}-\varepsilon _{1}\right) \left\Vert u\right\Vert _{p_{0}}^{p_{0}}-
\end{equation*}%
\begin{equation*}
c\left( \varepsilon \right) \left\vert \lambda \right\vert \left\Vert
b_{0}\right\Vert _{\infty }\left\Vert u\right\Vert _{\widetilde{p}}^{%
\widetilde{p}}-\left\vert \lambda \right\vert \underset{\Omega }{\int }%
b_{1}\left( x\right) \left\vert u\right\vert ^{p_{2}+1}\ dx-C_{\varepsilon
_{1}}\left( \left\vert \lambda \right\vert ,\mu ,\left\Vert a_{1}\right\Vert
_{q},\left\Vert b_{2}\right\Vert _{q}\right) .
\end{equation*}

since $p_{1}\leq p\left( 1-p_{0}^{-1}\right) $ by the conditions $\widetilde{%
p}\leq p_{0}$. Hence either of these cases take place: $\widetilde{p}<p_{0}$
and $p_{2}+1<p_{0}$ or one of equations $\widetilde{p}=p_{0}$ or $%
p_{2}+1=p_{0}$ holds, if $\widetilde{p}<p_{0}$ and $p_{2}+1<p_{0}$ then we
can estimate it by second term from right side of the previous inequality
with using Young inequality, and if $\widetilde{p}=p_{0}$ or $p_{2}+1=p_{0}$
then it is enough to choose number $\left\vert \lambda \right\vert $
sufficiently small. Thus we obtain the following inequality: 
\begin{equation*}
\left\langle f\left( u\right) ,u\right\rangle \geq \left( 1-\varepsilon
\left\vert \lambda \right\vert \left\Vert b_{0}\right\Vert _{\infty }\right)
\left\Vert \left\vert \nabla u\right\vert \right\Vert _{p}^{p}+\left( \mu
A_{0}-\varepsilon _{1}-\varepsilon _{2}\right) \left\Vert u\right\Vert
_{p_{0}}^{p_{0}}
\end{equation*}%
\begin{equation*}
-C_{\varepsilon _{1}}\left( \left\vert \lambda \right\vert ,\mu ,\varepsilon
_{2},\left\Vert a_{1}\right\Vert ,\left\Vert b_{1}\right\Vert ,\left\Vert
b_{2}\right\Vert \right) .
\end{equation*}

Consequently, inequality (2.1) of Theorem 1 is fulfilled, $\left\vert
\lambda \right\vert $ must be sufficiently small, i.e. the statement of
Theorem 5 is true since the operator 
\begin{equation*}
f:W_{0}^{1,p}\left( \Omega \right) \cap L^{p_{0}}\left( \Omega \right)
\longrightarrow W^{-1,q}\left( \Omega \right) +L^{q_{0}}\left( \Omega
\right) ,\ q_{0}=\frac{p_{0}}{p_{0}-1},
\end{equation*}%
is bounded by virtue of the obtained estimations here.
\end{proof}

\textbf{5.2. Everywhere solvability.} Now we reduce the conditions under
which imposed problem (5.1)-(5.2) is everywhere solvable. Let conditions
(5.3) - (5.5) be fulfilled and consider the following conditions 
\begin{equation*}
\left\vert G_{0}\left( x,\xi \right) -G_{0}\left( x,\xi _{1}\right)
\right\vert \leq c_{0}\left( x,\widetilde{\xi }\right) \ \left\vert \xi -\xi
_{1}\right\vert ,
\end{equation*}%
\begin{equation}
\left\vert G_{1}\left( x,\xi ,\eta \right) -G_{1}\left( x,\xi _{1},\eta
_{1}\right) \right\vert \leq c_{1}\left( x,\widetilde{\xi },\widetilde{\eta }%
\right) \ \left\vert \eta -\eta _{1}\right\vert ,  \tag{5.7}
\end{equation}%
hold for a.e. $x\in \Omega $, and any $\left( \xi ,\eta \right) $, $\left(
\xi _{1},\eta _{1}\right) \in \Re \times \Re ^{n}$, where $c_{0}\left( x,\xi
\right) $, $c_{1}\left( x,\xi ,\eta \right) $ are some Caratheodory
functions such that $\widetilde{\xi }=\widetilde{\xi }\left( \xi ,\xi
_{1}\right) $, $\widetilde{\eta }=\widetilde{\eta }\left( \eta ,\eta
_{1}\right) $ are continuous functions, and moreover $c_{1}\left( x,v,\nabla
v\right) $, $c_{0}\left( x,v\right) $ are bounded operators such that if $%
v\left( x\right) $ belongs to a bounded subset $D$ of $\ W_{0}^{1,p}\left(
\Omega \right) \cap L^{p_{0}}\left( \Omega \right) $, i.e. if $\left\Vert
v\right\Vert _{W_{0}^{1,p}\left( \Omega \right) \cap L^{p_{0}}\left( \Omega
\right) }\leq K_{0},$ then 
\begin{equation*}
\left\Vert c_{1}\left( x,v,\nabla v\right) \right\Vert _{L^{\infty }\left(
\Omega \right) }\leq K_{1},\quad \left\Vert c_{0}\left( x,v\right)
\right\Vert _{L^{\infty }\left( \Omega \right) }\leq K_{2},
\end{equation*}%
for some numbers $K_{0}$, $K_{1},K_{2}>0$ , i.e. 
\begin{equation}
c_{j}\left( x,\cdot ,\cdot \right) :W_{0}^{1,p}\left( \Omega \right) \cap
L_{p_{0}}\left( \Omega \right) \longrightarrow L^{\infty }\left( \Omega
\right) ,\ j=0,1,  \tag{5.8}
\end{equation}%
are bounded operators.

Then 
\begin{equation*}
\left\vert \left\langle G_{1}\left( x,u,\nabla u\right) -G_{1}\left(
x,v,\nabla v\right) ,u-v\right\rangle \right\vert \leq 
\end{equation*}%
\begin{equation*}
\left\Vert c_{1}\left( x,\widetilde{u}\left( u,v\right) ,\nabla \widetilde{u}%
\left( \nabla u,\nabla v\right) \right) \right\Vert _{\infty }\left\Vert
\nabla u-\nabla v\right\Vert _{p}\left\Vert u-v\right\Vert _{q},
\end{equation*}%
holds for any $u,v\in W_{0}^{1,p}\left( \Omega \right) \cap L^{p_{0}}\left(
\Omega \right) $. Hence we get 
\begin{equation*}
\left\langle f\left( u\right) -f\left( v\right) ,u-v\right\rangle \equiv
\left\langle \left\vert \nabla u\right\vert ^{p-2}\nabla u-\left\vert \nabla
v\right\vert ^{p-2}\nabla v,\nabla \left( u-v\right) \right\rangle +
\end{equation*}%
\begin{equation*}
\mu \left\langle G_{0}\left( x,u\right) -G_{0}\left( x,v\right)
,u-v\right\rangle +\lambda \left\langle G_{1}\left( x,u,Du\right)
-G_{1}\left( x,v,Dv\right) ,u-v\right\rangle \geq 
\end{equation*}%
\begin{equation*}
\widehat{C}_{0}\left\Vert \nabla \left( u-v\right) \right\Vert _{p}^{p}+\mu
\left\langle a\left( x\right) \left( \left\vert u\right\vert
^{p_{0}-2}u-\left\vert v\right\vert ^{p_{0}-2}v\right) ,u-v\right\rangle ,
\end{equation*}%
and we obtain the following inequality by using conditions (5.4), (5.7) and
(5.8) 
\begin{equation*}
\left\langle f\left( u\right) -f\left( v\right) ,u-v\right\rangle \geq 
\widehat{C}\left\Vert \nabla \left( u-v\right) \right\Vert _{p}^{p}+\widehat{%
A}_{0}\left\Vert u-v\right\Vert _{p_{0}}^{p_{0}}-
\end{equation*}%
\begin{equation*}
\left\Vert c_{1}\left( x,u,\nabla u,v,\nabla v\right) \right\Vert _{\infty
}\left\Vert \nabla u-\nabla v\right\Vert _{p}\left\Vert u-v\right\Vert _{q}.
\end{equation*}

Consequently, we have 
\begin{equation*}
\left\Vert f\left( u\right) -f\left( v\right) \right\Vert _{W_{q}^{-1}\left(
\Omega \right) }\cdot \left\Vert \nabla \left( u-v\right) \right\Vert
_{p}\geq \widehat{C}\left\Vert \nabla \left( u-v\right) \right\Vert _{p}^{p}+%
\widehat{A}_{0}\left\Vert u-v\right\Vert _{p_{0}}^{p_{0}}-
\end{equation*}%
\begin{equation*}
\left\Vert c_{1}\left( x,u,\nabla u,v,\nabla v\right) \right\Vert _{\infty } 
\left[ \varepsilon \left\Vert \nabla \left( u-v\right) \right\Vert
_{p}^{p}+c\left( \varepsilon \right) \left\Vert u-v\right\Vert _{q}^{q}%
\right] ,
\end{equation*}%
or 
\begin{equation*}
\left\Vert f\left( u\right) -f\left( v\right) \right\Vert _{W_{q}^{-1}\left(
\Omega \right) }\cdot \left\Vert \nabla \left( u-v\right) \right\Vert
_{p}\geq \widehat{C}_{1}\left\Vert \nabla \left( u-v\right) \right\Vert
_{p}^{p-1}+
\end{equation*}%
\begin{equation*}
\widehat{A}_{0}\left\Vert u-v\right\Vert _{p_{0}}^{p_{0}}-c\left(
\varepsilon \right) \left\Vert u-v\right\Vert _{q}^{q}.
\end{equation*}

Thus we get that the conditions of Corollary 1 are fulfilled, i.e. if we
continue this proof as in the section 4 then we obtain that the following
result is true.

\begin{theorem}
Let conditions (5.3)-(5.5), (5.7) and (5.8) be fulfilled and the numbers $%
\lambda ,\mu $ satisfy the conditions of Theorem 5, then problem (5.1)-(5.2)
is solvable in $W_{0}^{1,p}\left( \Omega \right) \cap L^{p_{0}}\left( \Omega
\right) $ for any $h\in W^{-1,q}\left( \Omega \right) $.
\end{theorem}

\begin{remark}
It should be noted that a remark similar to remark 2 takes place for the
problem investigated here. Moreover, we can obtain the same conclusion for
the considered general case. Therefore, for the investigation of the
spectrum of the nonlinear operators by using the fixed-point theorem
mentioned above, we need to consider some particular cases of these problems.
\end{remark}

\end{document}